\documentclass[11pt]{article}

\usepackage{amssymb, amsmath, amsthm, graphicx, xcolor}
\usepackage[left=1in,top=1in,right=1in]{geometry}
\usepackage{subfigure}
\usepackage{hyperref}

\date{}

\theoremstyle{plain}
      \newtheorem{theorem}{Theorem}[section]
      \newtheorem{lemma}[theorem]{Lemma}

      \newtheorem{observation}[theorem]{Observation}

      \newtheorem{conjecture}[theorem]{Conjecture}
\theoremstyle{definition}

\theoremstyle{remark}

\title{A note on the no-$(d+2)$-on-a-sphere problem}
\author{Andrew Suk\thanks{Department of Mathematics, University of California at San Diego, La Jolla, CA 92093. Email: asuk@ucsd.edu. Research supported by an NSF
CAREER Award and by NSF Awards DMS-1952786 and DMS-2246847.} \and 
Ethan Patrick White\thanks{Department of Mathematics, University of Illinois Urbana-Champaign, Urbana, IL 61801. Email: epw@illinois.edu. Research supported by an NSERC PDF Award.}}
\date{}

\begin{document}

\maketitle

\begin{abstract} 

For fixed $d\geq 3$, we construct subsets of the $d$-dimensional lattice cube $[n]^d$ of size $n^{\frac{3}{d + 1} - o(1)}$ with no $d+2$ points on a sphere or a hyperplane.  This improves the previously best known bound of $\Omega(n^{\frac{1}{d-1}})$ due to Thiele from 1995.

\end{abstract}

\section{Introduction}

The famous \emph{no-three-in-line problem}, raised by Dudeney in 1917 in the special case $n = 8$, asks if it is possible to select $2n$ points from the lattice square $[n]^2$ such that no three are collinear.  Clearly, one cannot do better than $2n$ as $[n]^2$ can be covered by $n$ lines.  For $n\leq 52$, many authors have published solutions to this problem obtaining the bound of $2n$ (e.g.~see \cite{Fl}).  However, for large values of $n$, the best known lower bound is due to Hall et al.~\cite{Ha} which contains at least $\left(\frac{3}{2} - \varepsilon\right)n$ points, for any $\varepsilon > 0$ and $n$ sufficiently large.

The similar \emph{no-four-on-a-circle} problem, raised by Erd\H os and Purdy in 1981 (see \cite{RKG}), asks to determine the maximum number of points that can be selected from $[n]^2$ such that no four are on a circle.  Here, a line is also considered to be a \emph{degenerate} circle.  Again, we have the trivial upper bound of $3n$ as any vertical line must contain at most three points.  This upper bound was improved by Thiele \cite{thiele1}, who showed that at most $\frac{5}{2}n - \frac{3}{2}$ points can be selected, and moreover, he gave a construction of $\frac{n}{4}$ points from $[n]^2$ with no four on a circle (or  a line).

In this paper, we study the no-four-on-a-circle problem in higher dimensions (Problem 4 in Chapter 10 of \cite{brass}).  A \emph{$k$-sphere} is a $k$-dimensional sphere.  Thus, a 0-sphere is a pair of points, a 1-sphere is a circle, and etc.  For simplicity, we will simply use the term \emph{sphere} when referring to a $(d-1)$-dimensional sphere in $\mathbb{R}^d$.  Again, the maximum number of points that can be selected from the $d$-dimensional lattice cube $[n]^d$ with no $d+2$ points on a sphere or a hyperplane is at most $(d+1)n$ since we can cover $[n]^d$ with $n$ hyperplanes.  In the other direction, Thiele \cite{thiele2} showed that one can select $\Omega(n^{\frac{1}{d-1}})$ points from $[n]^d$ with no $d+2$ points on a sphere or a hyperplane, providing the first non-trivial construction for this problem.  In this note, we prove the following.

\begin{theorem}\label{higherD} Let $d \geq 3$ be a positive integer. Then there is a subset of the $d$-dimensional lattice cube $[n]^d$ with $n^{\frac{3}{d+1} -o(1)}$ points with no $d+2$ members on a sphere or a hyperplane.
\end{theorem}

While it is possible that one can improve this bound to $\Omega(n)$, we make the following more modest conjecture.

\begin{conjecture}\label{mainconj}
    Let $d \geq 3$ be a positive integer. Then there is a subset of the $d$-dimensional lattice cube $[n]^d$ with $\Omega(n^{\frac{d}{d + 1}})$ points with no $d+2$ members on a sphere or a hyperplane.
\end{conjecture}

Our paper is organized as follows.  In Section \ref{secvc}, we recall several results from VC-dimension theory that will be used in the proof of Theorem \ref{higherD}.  In Section \ref{proofhigherD}, we prove Theorem \ref{higherD}.   We conclude our paper with several remarks and related problems.  For sake of clarity, we systemically remove floor and ceilings whenever they are not crucial.  All logarithms are in based 2 unless stated otherwise.  

\section{VC-dimension theory}\label{secvc}
In this section, we recall several results from VC-dimension theory that will be used in the proof of Theorem \ref{higherD}.   First we introduce some terminology.

Let $G = (P,Q,E)$ be a bipartite graph with independent sets $P,Q$ and edges $E \subset P \times Q$. For $q \in Q$ we define the neighbourhood $N_Q(q) = \{ p \in P \colon (p,q) \in E\}$. Define the set-system $\mathcal{F} = \{ N_G(q) \colon q \in Q\}$ with ground set $P$. We say a set $S \subset P$ is \emph{shattered} by $\mathcal{F}$ if for every subset $B \subset S$, there is a set $A \in \mathcal{F}$ such that $A \cap S = B$. The \emph{Vapnik-Chervonenkis dimension} (VC-dimension) of $(P,\mathcal{F})$ is the largest integer $d$ for which there exists a subset $S \subset P$, with $|S| = d$, that is shattered by $\mathcal{F}$. 

The \emph{primal shatter function} of $(P,\mathcal{F})$ is defined as
\[ \pi_{\mathcal{F}}(z) = \max_{P'\subset P, |P'| = z} | \{ A \cap P' \colon A \in \mathcal{F} \}|. \]

\noindent The well-known Sauer-Shelah lemma \cite{sauer,shelah} states that if $d_0$ is the VC-dimension of $\mathcal{F}$, then 

\begin{equation}\label{sauer} \pi_{\mathcal{F}}(z)  \leq \sum_{i=0}^{d_0} \binom{z}{i}. \end{equation}

\noindent We will need the following result due to Fox, Pach, Sheffer, Suk, and Zahl.

\begin{theorem}[\cite{fox}]\label{bipVC}

Let $G = (P, Q, E)$ be a bipartite graph with $|P| = m$ and $|Q| = n$, such that the set
system $\mathcal{F}_1 = \{N(q) \colon q \in Q\}$, with ground set $P$, satisfies $\pi_{\mathcal{F}_1}(z) \leq cz^{d_0}$ for all~$z$. Then if $G$ is $K_{t,t}$-free, we have 

\[ |E(G)| \leq c_1(mn^{1-1/d_0} + n), \]
where $c_1 = c_1(c, d_0, t)$.  In particular, $c_1 \leq 10\cdot ct^{2d_0 + 1}(d_0\log d_0)^{d_0}.$

\end{theorem}

A subset $A \subset [n]^d$ is a \emph{maximal spherical set of $[n]^d$} if all points in $A$ lie on a single sphere $S$ in $\mathbb{R}^d$, and no point of $[n]^d$ can be added to $A$ while keeping all points on the sphere $S$. Let $\mathcal{S}_{n,d}$ denote the set-system of all maximal spherical sets of $[n]^d$, with ground set $[n]^d$.

\begin{lemma}\label{sphereVC}

For $d\geq 2$, the VC-dimension of the set-system $([n]^d,\mathcal{S}_{n,d})$ is at most $d+1$. 

\end{lemma}

\begin{proof}
We proceed by induction on $d$.  The base case $d = 1$ is trivial since a 0-sphere consists of 2 points.  For the inductive step, assume the statement holds for $d' < d$.  For sake of contradiction, suppose that a set $Q$ of $d+2$ points in $[n]^d$ is shattered by $\mathcal{S}_{n,d}$. 

\medskip

\noindent \emph{Case 1.}  Suppose $Q$ is in general position, that is, no $d+1$ members in $Q$ lie on a hyperplane.  By linear algebra, any $d+1$ points in $Q$ determines a unique sphere in $\mathbb{R}^d$. 
 However, since $Q$ is shattered by $\mathcal{S}_{n,d}$, there is a sphere that contains $Q$.  Hence, there is no sphere that contains exactly $d + 1$ points from $Q$.  Contradiction.

 \medskip
\noindent \emph{Case 2.}  Suppose $Q$ contains a $(d+1)$-tuple $Q'$ that lies on a hyperplane $h$ in $\mathbb{R}^d$.  Since the intersection of a sphere with $h$ is a $(d-2)$-sphere in $h$, by induction on $d$, $Q'$ cannot be shattered by $\mathcal{S}_{n,d}$.  Hence, $Q$ cannot be shattered by  $\mathcal{S}_{n,d}$, which is a contradiction. \end{proof}

\section{Proof of Theorem~\ref{higherD}} \label{proofhigherD}
 
In this section, we prove Theorem \ref{proofhigherD}.  First, let us recall several results.  The first is the well-known Chernoff inequality (see \cite[Theorem 2.8]{janson}).

\begin{lemma}[Chernoff's inequality]\label{chernoff}

    Let $X_1,\ldots, X_n$ be independent random variables such that $\mathbb{P}(X_i = 1) = q$ and $\mathbb{P}(X_i = 0) = 1-q$, and let $X = \sum_i X_i$. Then for $0 < \delta < 1$, we have

$$\mathbb{P}(X \geq (1 + \delta)qn) \leq e^{-\frac{\delta^2qn}{3}},$$and
$$\mathbb{P}(X \leq (1 - \delta)qn) \leq e^{-\frac{\delta^2qn}{3}}.$$

\end{lemma}

Next, we estimate the number of points from $[n]^d$ that can lie on a sphere or a hyperplane.  For the latter, we will use the following result due to Balogh and White.

\begin{lemma}[\cite{balogh}]\label{hyperplaneBound} 
Let $d \geq 2$ be a positive integer, and $a_1,\ldots,a_d \in \mathbb{Z}$ not be all zero with greatest common denominator 1. Let $s = \max_i\{|a_i|\}$, $n \geq s$ be a positive integer, and set
\[ \mathcal{L} = \{ (x_1,\ldots,x_d) \in \mathbb{Z}^d \colon \sum_{i=1}^d a_ix_i = 0 \}.\]
If $\mathcal{L} \cap [n]^d$ spans a $(d-1)$-dimensional subspace, then $|\mathcal{L} \cap [n]^d| \leq 3^dn^{d-1}/s$. 
    
\end{lemma}

\begin{lemma}\label{cohyperplaneCount}

The number of $(d+2)$-tuples in $[n]^d$ that lie on a common hyperplane is at most $O(n^{d^2+d-1})$. 
    
\end{lemma}

\begin{proof} By the lemma above, the number of $(d+1)$-tuples in $[n]^d$ that lie on a hyperplane which passes through the origin is at most

    \[ \sum_s \left( \frac{n^{d-1}}{s} \right)^{d+1} s^{d-1} = O(n^{d^2 - 1}). \]

\noindent By symmetry, for any fixed point $p\in [n]^d$, the number of $(d+1)$-tuples in $[n]^d$ that lies on a hyperplane which passes through $p$ is at most  $O(n^{d^2 - 1}).$  After summing over all points in $[n]^d$, the statement follows. Note that Lemma~\ref{cohyperplaneCount} applies to hyperplanes whose intersection with $[n]^d$ spans a $(d-1)$-dimensional subspace. We can assume that any $(d+2)$-tuple in $[n]^d$ lying on a common hyperplane meets this requirement by adding other points from $[n]^d$ if needed.\end{proof}

Finally, we will use the following result due to Sheffer that bounds the number of points from $[n]^d$ that lies on a $k$-sphere in $\mathbb{R}^d$, for $k \leq d-1$.  See also \cite[Sec.~11.2, Equation~11.9]{iwaniec}, \cite[Theorem~2]{shel}.

\begin{lemma}[\cite{sheffer}, Lemma 3.2]\label{simplepoorCircles} Let $d \geq 2$ and $k\leq d-1$. 
 Then there is a positive constant $c = c(d)$ such that every $k$-sphere $\mathbb{R}^d$ contains at most $n^{k-1 + c/\log\log n }$ points from $[n]^d$.  
    
\end{lemma}

\medskip

Before we prove Theorem \ref{higherD}, let us give a brief outline of the argument.  First, we use the probabilistic method and Lemmas \ref{cohyperplaneCount} and \ref{simplepoorCircles} to obtain a large subset $A\subset [n]^d$ such that very few members of $A$ lie on a $(d-2)$-sphere or a hyperplane in $\mathbb{R}^d$.  We then apply Theorem \ref{bipVC} and ideas from incidence geometry to estimate the number of $r$-tuples in $A$ that lie on a common sphere.  Finally, we apply the deletion method to find a large subset $A'' \subset A$ such that no $d+2$ members of $A''$ lie on a sphere or a hyperplane.  We now flesh out all of the details of the proof.
\medskip

\begin{proof}[Proof of Theorem \ref{higherD}]
    Without loss of generality, we can assume that $n$ is a power of 2.  Indeed, otherwise we can find an integer $n' < n$ that is a power of 2 such that $n' > n/2$, and apply the arguments below to the subcube $[n']^d$.  This would only change the hidden constant in the $o(1)$ term in the exponent.  Moreover, we will assume that $n$ is sufficiently large.

Fix a positive integer $D < n$ that will be determined later, such that $D$ divides $n$. We partition the $d$-dimensional cube $[n]^d$ into $D^d$ smaller cubes $Q_1,\ldots, Q_{D^d}$, where each $Q_i$ is of the form

    $$\left\{i_1 \frac{n}{D}  + 1,\ldots, (i_1 + 1) \frac{n}{D} \right\}\times \cdots \times \left\{i_d\frac{n}{D} + 1,\ldots, (i_d + 1) \frac{n}{D} \right\}.$$
    
\noindent where $i_1,\ldots, i_d \in \{0,\ldots, D-1\}$. Hence, $|Q_i| = (n/D)^d$.

Consider a random subset $A\subset [n]^d$ obtained by selecting each point in $[n]^d$ independently with probability $n^{3-d}$.   Set $P_i = Q_i\cap A$, for $i = 1,\ldots, D^d$. Let $\mathcal{W}$ be the event that the subset $A\subset [n]^d$ satisfies the following properties.

\begin{enumerate}
    \item $n^3/2 \leq |A| \leq 2n^3$.
    \item $n^3D^{-d}/2 \leq |P_j| \leq 2n^3D^{-d}$ for all $1 \leq j \leq D^d$.
    \item Every $(d-2)$-sphere contains less than $n^{c_1/\log\log n}$ points of $A$, where $c_1 = c_1(d)$.
    \item The number $(d+2)$-tuples in $A$ that lie on a common hyperplane is at most $c_2n^{2d+5}$, where $c_2 = c_2(d)$.
\end{enumerate}

By Lemma \ref{chernoff} and Markov's inequality, event $\mathcal{W}$ holds with probability at least 1/2.   Indeed, by Lemma \ref{chernoff}, the probability that $|A| > 2n^3$ or $|A| < n^3/2$ is at most $e^{-\frac{n^3}{3}} + e^{-\frac{n^3}{12}}.$ Thus, the first property holds with high probability, that is, with probability tending to 1 as $n$ approaches infinity.  A similar argument follows for the second property.  For the third property, let us fix a $(d-2)$-sphere $S$ in $\mathbb{R}^d$.  By Lemma~\ref{simplepoorCircles}, there is a constant $c= c(d)$ such that $S$ contains at most $n^{d-3 + c\log\log n}$ points from $[n]^d$.  By~Lemma~\ref{chernoff}, 

$$\mathbb{P}[|S\cap A| \geq 2n^{c/\log\log n}] \leq  e^{-\frac{n^{c/\log\log n}}{3}}.$$

\noindent   Since there are at most $n^{d^2}$ $(d-2)$-spheres in $\mathbb{R}^d$ with at least $d$ points from $[n]^d$, the union bound implies that the third property holds with high probability.  For the fourth property, let $X$ denote the number $(d+2)$-tuples in $A$ that lie on a common hyperplane.  By Lemma \ref{cohyperplaneCount}, we have $\mathbb{E}[X] \leq c'n^{2d + 5}$, where $c' = c'(d)$. By Markov's inequality (see \cite{alon}), we have

$$\mathbb{P}[X> 10c'n^{2d + 5}] < \frac{\mathbb{E}[X]}{10c'n^{2d + 5}} < \frac{1}{10}.$$

 \noindent Putting everything together, and setting $c_1, c_2$ sufficiently large, event $\mathcal{W}$ holds with probability at least $1/2$.

 Thus, let us fix $A\subset [n]^d$ with the four properties described above, and set $t = 2n^{c_3/\log\log n}$, where $c_3 = \max\{c_1,c_2\}$.  Let $\mathcal{S}$ be the collection of spheres in $\mathbb{R}^d$, such that each sphere in $\mathcal{S}$ contains at least $d+1$ points from $A$ in general position.  Hence, $|\mathcal{S}| = O(|A|^{d + 1})$. Let $\mathcal{S}_j \subset \mathcal{S}$ be the set of spheres in $\mathcal{S}$ that contains at least one point of $P_j  = Q_j\cap A$. Let $G_j = (P_j,\mathcal{S}_j,E_j)$ be the bipartite incidence graph between $P_j$ and $\mathcal{S}_j$.  Since the intersection of two distinct spheres in $\mathbb{R}^d$ is a $(d-2)$-sphere, and every $(d-2)$-sphere contains less than $t$ points from $A$, each graph $G_j$ is $K_{t,2}$-free.

 Let $\mathcal{F}_{j}$ be the set system whose ground set is $P_{j}$, and whose sets are $S\cap P_{j}$, where $S \in \mathcal{S}_{j}$ and $|S\cap P_{j}| \geq t$.  That is,

$$\mathcal{F}_{j} = \{S\cap P_{j}: S\in \mathcal{S}_{j}, |S\cap P_j| \geq t\}.$$

\noindent By Lemma~\ref{sphereVC}, the VC-dimension of $\mathcal{F}_j$ is at most $d+1$.  By inequality \eqref{sauer}, we have $$\pi_{\mathcal{F}_j}(z) = O(z^{d+1}).$$ Hence, we apply Lemma \ref{bipVC} with $t = n^{c_3/\log\log n}$ to conclude that  $$|E_j| \leq n^{c_4/\log\log n}\left(|P_j||\mathcal{S}_j|^{\frac{d}{d+1}} + |\mathcal{S}_j|\right),$$ where $c_4 = c_4(d)$.  

Given the partition $[n]^d = Q_1\cup \cdots \cup Q_{D^d}$ described above, we say that a sphere $S$ \emph{crosses} the subcube $Q_i$ if $S\cap Q_i\neq \emptyset$.  

\begin{observation}
For fixed $d \geq 2$, every sphere $S$ in $\mathbb{R}^d$ crosses at most $c_5D^{d-1}$ subcubes $Q_i$, where $c_5 = c_5(d)$.
\end{observation}

\begin{proof}
   Let $c_5 = c_5(d)$ be a large constant that depends only on $d$.  We will determined $c_5$ later.  For sake of contradiction, suppose there is a sphere $S$ in $\mathbb{R}^d$ the crosses more than $c_5D^{d-1}$ subcubes $Q_i$. Note that we will not make any serious attempts to optimize the constant $c_5$.
   
 Let $T \subset [n]^d\cap S$ be the subset of points $[n]^d\cap S$ obtained by selecting 1 point from each subcube $Q_i$ that $S$ crosses.  Hence, $|T| \geq c_5D^{d-1}.$ Let $\mathcal{G} = (T,E)$ be an auxiliary graph on $T$, where two points in $T$ are adjacent if their distance is less than $n/D$.  Since two points in $T$ have distance less than $n/D$ only if they lie in adjacent subcubes $Q_i$ and $Q_j$, this implies that $\mathcal{G}$ has maximum degree $3^d - 1$.  By Turan's theorem, we can find a subset $T'\subset T$, such that the minimum distance between any two points in $T'$ is at least $n/D$.  Moreover, $|T'| \geq (c_5/3^d)D^{d-1}$.  

For each point $p \in T'$, consider the spherical cap $C_p$ obtained by intersecting $S$ with the ball $B_p$ centered at $p$ with radius $n/(2D)$.  Since the minimum distance among the points in $T'$ is at least $n/D$, the collection of spherical caps corresponding to the points in $T'$ will have pairwise disjoint interiors.  Moreover, since each spherical cap arises from a ball whose center is on $S$ with radius $n/(2D)$, each spherical cap will have area at least $c_6(n/D)^{d-1}$ where $c_6 = c_6(d)$.   Hence, the total area of all of the spherical caps corresponding to the points in $T'$ is at least

$$|T'|c_5(n/D)^{d-1} \geq \frac{c_5c_6}{3^d}n^{d-1}.$$

   \noindent On the other hand, all of the spherical caps on $S$ will lie inside of a ``large" ball $B$ with radius $dn$. Hence, the area of all of the spherical caps is at most the surface area of the ball $B$ with radius $dn$.   Since the surface area of $B$ is at most $c_7n^{d-1}$, where $c_7 = c_7(d)$, we have

   $$c_7n^{d-1} \geq \frac{c_5c_6}{3^d}n^{d-1}.$$

\noindent By setting $c_5$ sufficiently large, we have a contradiction.\end{proof}

By the observation above, $\sum_{j=1}^{D^d} |\mathcal{S}_j| = O(D^{d-1}|\mathcal{S}|)$. Putting all of these bounds together and applying Jensen's inequality, the number of incidences between $A$ and $\mathcal{S}$ is at most

    \begin{align*} \sum_{j=1}^{D^d} |E_j| &\leq n^{c_4/\log\log n}\sum_{j=1}^{D^d}  \left( |P_j||\mathcal{S}_j|^{\frac{d}{d+1}} + |\mathcal{S}_j| \right)\\
     &  \leq 2 n^{3 + c_4/\log\log n}D^{-d}\sum_{j=1}^{D^d}|\mathcal{S}_j|^{\frac{d}{d+1}} + n^{c_4/\log\log n}\sum_{j=1}^{D^d}|\mathcal{S}_j|  \\
    & \leq 2 n^{3 + c_4/\log\log n}D^{-d} D^{\frac{d}{d+1}}\left( \sum_{j=1}^{D^d} |\mathcal{S}_j| \right)^{\frac{d}{d+1}}  + c_5D^{d-1} n^{ c_4/\log\log n} |\mathcal{S}| \\
    & \leq 2 c_5^{\frac{d}{d+1}}n^{3 + c_4/\log\log n}D^{-\frac{d}{d+1}} |\mathcal{S}|^{\frac{d}{d+1}} + c_5D^{d-1} n^{ c_4/\log\log n} |\mathcal{S}|.
    \end{align*}

    Set $D$ to be a power of 2 such that

    $$   n^{\frac{3(d+1)}{d^2+d-1}} |\mathcal{S}|^{-\frac{1}{d^2+d-1}} \leq D \leq 2 n^{\frac{3(d+1)}{d^2+d-1}} |\mathcal{S}|^{-\frac{1}{d^2+d-1}}.$$ Note that $|\mathcal{S}| \leq |A|^{d+1}$, and therefore $1 < D < n$.  From above, the number of incidences between $\mathcal{S}$ and $A$ is at most$$n^{\frac{3(d^2-1)}{d^2+d-1}+c_8/\log\log n} |\mathcal{S}|^{\frac{d^2}{d^2+d-1}},$$ where $c_8 = c_8(d)$.  Let $\mathcal{S}_r\subset \mathcal{S}$ denote the set of $r$-rich spheres in $\mathcal{S}$, that is, the set of spheres in $\mathcal{S}$ with at least $r$ points from $A$. Then we have $$r|\mathcal{S}_r| \leq n^{\frac{3(d^2-1)}{d^2+d-1}+c_8/\log\log n} |\mathcal{S}_r|^{\frac{d^2}{d^2+d-1}},$$
    
    \noindent which implies 
    
$$|\mathcal{S}_r| \leq n^{3(d+1)+c_8/\log\log n} r^{-\frac{d^2+d-1}{d-1}}.$$

Let $\mathcal{T}_i\subset \mathcal{S}$ be the set of spheres in $\mathcal{S}$ with at least $2^i$ points from $A$, and at most $2^{i +1}$ from $A$.  Hence

$$|\mathcal{T}_i| \leq |\mathcal{S}_{2^i}| = n^{3(d+1)+c_8/\log\log n} 2^{-i\frac{d^2+d-1}{d-1}}.$$

\noindent Therefore, the number of $(d+2)$-tuples of $A$ that lie on a common sphere $S \in \mathcal{S}$ is at most

$$\sum\limits_{i = 1}^{\log |A|} |\mathcal{T}_i|\binom{2^{i + 1}}{d+2} \leq n^{3(d+1)+c_8/\log\log n}\sum\limits_{i = 1}^{\log |A|}  2^{-i\frac{1}{d-1}} \leq n^{3(d+1)+c_8/\log\log n}. $$

We now consider a random subset $A'\subset A$ obtained by selecting each point in $A$ independently with probability $p$, where $p$ will be determined later. Let $X$ be the number of $(d+2)$-tuples of $A'$ that lie on a common sphere or a hyperplane.  Then by linearity of expectation, we have

$$\mathbb{E}[X] \leq n^{3(d+1)+c_8/\log\log n}p^{d+2} + O(n^{2d + 5})p^{d+2} \leq n^{3(d+1)+c_9/\log\log n}p^{d+2},$$

\noindent and

$$\mathbb{E}[|A'|] = p|A|\geq n^3p/2,$$where $c_9 = c_9(d)$.  Hence, there is a constant $c_{10} = c_{10}(d)$ such that for $p = n^{-3d/(d+1) - c_{10}/\log\log n}$, we have $\mathbb{E}[X] < \mathbb{E}[|A'|]/4$.  By Lemma \ref{chernoff}, there is a subset $A'\subset A$ of size at least $p|A|/2$, such that $A'$ contains at most $p|A|/4$ $(d+2)$-tuples on a sphere or a hyperplane.  By deleting one point from each $(d+2)$-tuple on a sphere or a hyperplane, we obtain a subset $A''\subset A'$ of size at least
$$|A|p/4 \geq n^{\frac{3}{d+ 1} - 2c_{10}/\log\log n} = n^{\frac{3}{d+ 1} - o(1)},$$

\noindent such that no $d+2$ members in $A''$ lie on a sphere or a hyperplane.\end{proof}

\section{Concluding remarks}

Several authors \cite{roth,Le,BK} observed that if $n$ is prime, then by selecting $n$ points from $[n]^d$ on the ``modular moment curve" $(x,x^2,\ldots, x^d) \mod n$, for $x = 1,\ldots, n$, no $d+1$ points will lie on a hyperplane, and moreover, no $2d+1$ points lie on a sphere.  Unfortunately, this construction may contain $(d+2)$-tuples on a sphere.  Nevertheless, one can make the following simple observation.

\begin{theorem}\label{mcurve} Let $d \geq 2$ be a fixed positive integer. Then there is a subset of the $d$-dimensional lattice cube $[n]^d$ of size $\Omega(n)$ with no $d + 1$ points on a hyperplane, and no $2d$ points on a sphere.
\end{theorem}

\begin{proof} 

Let $n$ be prime, and let $A\subset [n]^d$ be points from the $d$-dimensional cube on the moment curve $(x,x^2,\ldots, x^d) \mod n$, for $1 \leq x \leq \lfloor n/(4d)\rfloor$. Hence, $|A| = \lfloor n/(4d)\rfloor$.  For sake of contradiction, suppose there is a sphere $S$ in $\mathbb{R}^d$, such that $S$ contains $2d$ points from $A$. Note that this means $S$ contains $2d$ points of $[n]^d$ when viewed as a sphere in $\mathbb{F}_n^d$. Let the center of $S \subset \mathbb{F}_n^d$ be $(c_1,\ldots, c_d)$ and radius $r$, such that $S$ contains $2d$ points from $A$. Then we have $2d$ solutions $s_1,s_2,\ldots, s_{2d}$ to the equation

$$(x-c_1)^2 + (x^2 - c_2)^2 + \cdots (x^d - c_d)^2 - r^2 = 0 \mod n.$$

    \noindent On the other hand, by the division algorithm, we have

    $$(x-c_1)^2 + (x^2 - c_2)^2 + \cdots (x^d - c_d)^2 - r^2  = (x - s_1)(x-s_2)\cdots (x-s_{2d})  \mod n,$$

    \noindent which implies that $s_1 + \cdots +s_{2d} = 0 \mod n$ as the coefficient of $x^{2d-1}$ is $0\mod n$.  However, $s_i \leq  n/(4d)$, which implies $s_1 + \cdots +s_{2d} \neq 0 \mod n$, contradiction.  Hence, no $2d$ points in $A$ lie on a sphere and no $d+1$ points lie on a hyperplane.  If $n$ is not prime, we apply Bertrand's Postulate to obtain a prime $n' < n$ such that $n' > n/2$, and apply the argument above to $[n']^d$.\end{proof}


Another natural question is to determine the maximum number of points that can be selected from $[n]^d$ with no $d+2$ points on a sphere, but allowing many points on a hyperplane.  Since each point from $[n]^d$ lies on a sphere centered at the origin with radius $\sqrt{t}$, where $t  = 1,2,\ldots, dn^2$, we have an upper bound of $d(d+1)n^2$ for this problem. Using Lemma \ref{simplepoorCircles} and the probabilistic method, one can show the existence of $n^{2 - 4/(d+1) - o(1)}$ points from $[n]^2 \subset [n]^d$ with no $d+2$ on a sphere.

\end{document}